\documentclass{amsart}
\usepackage{graphicx}
\usepackage{mathabx}
\usepackage{comment}
\usepackage{hyperref}

\usepackage[dvipsnames,x11names,svgnames]{xcolor} %UUS!
\definecolor{celadon}{rgb}{0.67, 0.88, 0.69}
\colorlet{kollane}{LightGoldenrod1}
\colorlet{roheline}{celadon}
\colorlet{sinine}{LightSkyBlue2}
\usepackage[framemethod=tikz]{mdframed}

\newcommand{\N}{{\mathbb N}}
\newcommand{\R}{{\mathbb R}}

\DeclareMathOperator{\Lip}{Lip}
\DeclareMathOperator{\spanop}{span}

\newtheorem{thm}{Theorem}[section]

\newtheorem{cor}[thm]{Corollary}

\newtheorem{lemma}[thm]{Lemma}
\newtheorem{prop}[thm]{Proposition}
\theoremstyle{definition}
\newtheorem{defn}[thm]{Definition}
\newtheorem{ex}[thm]{Example}
\theoremstyle{remark}
\newtheorem*{rem}{Remark}

\numberwithin{equation}{section}
\newcommand{\eps}{\varepsilon}

\begin{document}

\title[Diameter two properties for spaces of Lipschitz functions]{Diameter two properties for spaces of\\ Lipschitz functions}%
\author{Rainis Haller}%
\author{Andre Ostrak}%
\author{M\"{a}rt P\~{o}ldvere}%
\address{Institute of Mathematics and Statistics, University of Tartu, Narva mnt 18, 51009, Tartu, Estonia}%
\email{rainis.haller@ut.ee, andre.ostrak@ut.ee, mart.poldvere@ut.ee}%

\thanks{}%

%\date{}%
%\dedicatory{}%
%\commby{}%
% ----------------------------------------------------------------
\begin{abstract}
We solve some open problems regarding diameter two properties within the class of Banach spaces of real-valued Lipschitz functions by using the de Leeuw transform. Namely, we show that: the diameter two property, the strong diameter two property, and the symmetric strong diameter two property are all different for these spaces of Lipschitz functions; the space $\Lip_0(K_n)$ has the symmetric strong diameter two property for every $n\in \N$, including the case of $n=2$; every local norm-one Lipschitz function is a Daugavet point.
\end{abstract}
\maketitle
% ----------------------------------------------------------------

\section{Introduction}
Let $X$ be a real nontrivial Banach space. 
We denote the closed unit ball, the unit sphere, and the dual space of $X$ by $B_X$, $S_X$, and $X^*$, respectively. 
A \emph{slice} of $B_{X}$ is a set of the form
\[
S(x^*, \alpha)\coloneq \{x\in B_{X}\colon x^*(x)>1-\alpha\},
\]
where $x^*\in S_{X^*}$ and $\alpha>0$. 
If $X$ is a dual space, 
then slices whose defining functional comes from  (the canonical image of) the predual of $X$ are called weak$^*$ slices.

According to the terminology in \cite{MR3098474} and \cite{MR4023351} (see also \cite{MR3334951}), the Banach space $X$ has the
\begin{itemize}\label{def: d2p-s}
\item \emph{slice diameter $2$ property} (briefly, \emph{slice-D$2$P}) if every slice of $B_X$ has dia\-meter $2$;
\item \emph{diameter $2$ property} (briefly, \emph{D$2$P}) if every nonempty relatively weakly open subset
of $B_X$ has diameter $2$;
\item \emph{strong diameter $2$ property} (briefly, \emph{SD$2$P}) 
if every convex combination of slices of $B_X$ has diameter $2$, 
i.e., the diameter of $\sum_{i=1}^n \lambda_i S_i$ is $2$ 
whenever $n\in\mathbb N$, $\lambda_1,\dotsc,\lambda_n\geq 0$ with $\sum_{i=1}^n\lambda_i=1$, 
and $S_1,\dotsc,S_n$ are slices of $B_X$;
\item\emph{symmetric strong diameter $2$ property} (briefly, \emph{SSD$2$P}) if, 
for every $n\in\mathbb N$, every family $\{S_1,\dotsc, S_n\}$ of slices of $B_X$, and every $\varepsilon>0$, 
there exist $f_1\in S_1,\dotsc, f_n\in S_n$, and $g\in B_{X}$ with $\|g\|>1-\varepsilon$ 
such that $f_i\pm g\in S_i$ for every $i\in \{1,\ldots,n\}$.
\end{itemize}
If $X$ is a dual space, then we also consider the weak$^*$ versions of these diameter two properties ($w^
*$-slice D$2$P, $w^*$-D$2$P, $w^*$-SD$2$P, and $w^*$-SSD$2$P), where slices and weakly open subsets in the above definitions are replaced by weak$^*$ slices and weak$^*$ open subsets, respectively. 

In this paper we study diameter two properties in the space $\Lip_0(M)$.
Let $M$ be a pointed metric space, that is, a metric space with a fixed point $0$. The space $\Lip_0(M)$ is the Banach space of all Lipschitz functions $f\colon M\to \mathbb{R}$ with $f(0)=0$ equipped with the norm
\[
\|f\| =\sup\left\{ \frac{|f(x)-f(y)|}{d(x,y)}\colon x,y\in M, x\neq y\right\},
\]
i.e., $\|f\|$ is the Lipschitz constant of $f$.
The (sub)space $\mathcal{F}(M)\coloneq \overline{\spanop}\left\{\delta_x \colon x\in M \right\}$ of $\Lip_0(M)^*$ is called the \emph{Lipschitz-free space over $M$,} where $\langle f, \delta_x\rangle=f(x)$ for every $f\in\Lip_0(M)$. It can be shown that, under this duality, $\mathcal{F}(M)^*$ is isometrically isomorphic to $\Lip_0(M)$.

Diameter two properties of $\Lip_0(M)$ have been studied
in \cite{zbMATH05159098}, \cite{MR3803112}, \cite{MR3917942}, \cite{MR4233633}, \cite{MR3985517},
and \cite{MR4093788}.
In \cite{zbMATH05159098}, Ivakhno proved that if a metric space $M$ is unbounded or not uniformly discrete, then the space $\Lip_0(M)$ has the slice-D$2$P. Recall that $M$ is said to be \emph{uniformly discrete} if $\inf\{d(x,y)\colon x,y\in M,\, x\neq y\}>0$.
In \cite{MR3803112}, Procházka and Rueda Zoca  introduced a property of metric spaces that they called
the \emph{long trapezoid property} (briefly, \emph{LTP};  see the definition on page~\pageref{def: LTP} below),
and proved that $\Lip_0(M)$ has the $w^\ast$-SD$2$P if and only if $M$ has the LTP.
In \cite{MR3917942}, Haller et al. proved that, for every $n\in\mathbb N$, the space $\Lip_0(K_n)$ has the $w^\ast$-SSD$2$P
(recall that the metric space $K_n$ is the metric subspace of the space $\ell_\infty$
of sequences with terms in $\{0,1,\dotsc,n\}$).
In \cite{MR4233633}, Ostrak introduced a property of metric spaces that he called
the \emph{strong long trapezoid property} (briefly, \emph{SLTP}; see the definition on page~\pageref{def: SLTP} below),
and proved that $\Lip_0(M)$ has the $w^\ast$-SSD$2$P if and only if $M$ has the SLTP.
He also gave an example of a metric space with the LTP but without the SLTP
thus showing that the $w^\ast$-SD$2$P and the $w^\ast$-SSD$2$P for $\Lip_0(M)$ are different properties.
Note that unbounded metric spaces and not uniformly discrete metric spaces as well as the spaces $K_1, K_2, \dotsc$ all have the SLTP.
In \cite{MR3985517}, Cascales et al. proved that if a metric space $M$ has infinitely many cluster points or $M$ is discrete but not uniformly discrete, then the space $\Lip_0(M)$ has even the SSD$2$P. In \cite{MR4093788}, Langemets and Rueda Zoca generalised this result by proving that the same is true if $M$ is unbounded or not uniformly discrete. 
 
\begin{thm}\label{thm: If M is unbounded or not uniformly discrete, then Lip_0(M) has the SSD2P}%[see {\cite[Theorem~2.2]{MR4093788}}]
If the metric space $M$ is unbounded or not uniformly discrete, then the space $\Lip_0(M)$ has the SSD$2$P.
\end{thm}
\noindent
Theorem \ref{thm: If M is unbounded or not uniformly discrete, then Lip_0(M) has the SSD2P} leaves open 
for which bounded but not uniformly discrete metric spaces $M$ the space $\Lip_0(M)$
has the SSD$2$P (or SD$2$P or D$2$P or slice-D$2$P).
In \cite{MR4093788}% [Propositions~2.7 and 2.8]
, Langemets and Rueda Zoca proved that 
the space $\Lip_0(K_n)$ has the SSD$2$P whenever $n\in\N\setminus\{2\}$. 
% This answered a question posed by Ivakhno \cite[p 114]{zbMATH05159098}, 
% but left open whether the space $\Lip_0(K_2)$ has the (S)SD$2$P (see \cite[the discussion after the proof of Proposition~2.7]{MR4093788}).
In this paper (see Theorem~\ref{thm: SSD2P from seq-SLTP} and Proposition~\ref{prop: MR-l K_n on alati seq-SLTP} below), we show that, in fact, this is true for \emph{every} $n\in\N$, including the case of $n=2$.

\begin{thm}\label{thm: ruumil Lip_0(K_n) on alati SSD2P}%
For every $n\in\N$, the space $\Lip_0(K_n)$ has the SSD2P.
\end{thm}

% Our proofs of Theorems \ref{thm: Lip_0(K_2) has the SS2P}
% and~\ref{thm: two classes of metric spaces for which Lip_0(M) has the SSD2}
% (as well as the proof of Theorem \ref{thm: If f in Lip_0(M) is local, then f is a Daugavet point} below)
% heavily rely on the following well known\kollane{(?)} observation.
% Assume that $M$ is a pointed metric space.

\noindent 
% As previously stated, it has been shown that the $w^\ast$-SD$2$P and the $w^\ast$-SSD$2$P for $\Lip_0(M)$ are different properties. 
In most cases, it seems to be unknown whether for $\Lip_0(M)$ the above-mentioned diameter two properties differ from each other.
%The differences between the above-mentioned diameter two properties have so far been unknown for $\Lip_0(M)$. 
E.g., in \cite[Introduction]{MR4093788}, the authors say that 
it is not known whether the slice-D$2$P implies the SSD$2$P within the class of spaces of Lipschitz functions, 
and it is not known whether the SSD$2$P and the $w^*$-SSD$2$P coincide in general. 
Our Example~\ref{ex: seq-LTP but not SLTP}, combined with Theorem~\ref{thm: SD2P from seq-LTP}, 
shows that the SD$2$P does not follow from $w^*$-SSD$2$P for the spaces of Lipschitz functions. 
Our Example~\ref{ex: D2P but not LTP}, combined with Lemma~\ref{lem: main lemma for detecting the D2P}, 
shows that the D$2$P does not follow from the $w^*$-SD$2$P for the spaces of Lipschitz functions. 
Therefore, we have the following.

\begin{thm}
The SSD$2$P, the SD$2$P, and the D$2$P are three different properties for the spaces of Lipschitz functions. In fact, the $w^\ast$-SSD$2$P and the SD$2$P are different, and the $w^\ast$-SD$2$P and the D$2$P are different for the spaces of Lipschitz functions.
\end{thm}
\noindent

\noindent It remains open whether, for the spaces of Lipschitz functions, any of the above-mentioned (non-weak$^\ast$) diameter two properties coincides with its weak$^\ast$ version. In fact, we don't even know if, for the spaces of Lipschitz functions, the $w^*$-SSD$2$P implies the slice-D$2$P. It also remains open whether there exists a metric space $M$ such that $\Lip_0(M)$ has the slice-D$2$P but not the D$2$P.

In \cite{MJAR}, Jung and Rueda Zoca studied Daugavet points in $\Lip_0(M)$ in connection with locality properties of Lipschitz functions. They posed and addressed the question of whether every local norm-one $f$ in $\Lip_0(M)$ is a Daugavet point. We answer that question with the following theorem.

\begin{thm}[cf. {\cite[Proposition 3.4 and Theorem 3.6]{MJAR}}]\label{thm: If f in Lip_0(M) is local, then f is a Daugavet point}
Let $M$ be a pointed metric space. If $f\in S_{\Lip_0(M)}$ is local,
then $f$ is a Daugavet point.
\end{thm}

In general, one faces difficulties when dealing with the dual space $\Lip_0(M)^*$ due to the lack of a useful characterisation of the space. Our results will heavily rely on the following observation. 
Set
\begin{equation}\label{eq: Gamma:=(M times M) setminus ((x,x): x in M)}
\widetilde{M}=(M\times M)\setminus \{(x,x)\colon x\in M\}.
\end{equation}
Then the mapping, also called the \emph{de Leeuw's transform} (see, e.g., \cite{MR3792558}), 
\[
\Lip_0(M)\ni f\longmapsto \widetilde{f}\in\ell_{\infty}(\widetilde{M}),
\quad\text{where}\quad
\widetilde{f}(x,y)=\frac{f(x)-f(y)}{d(x,y)},
\]
is linear and isometric. Recall that the dual space of $\ell_{\infty}(\widetilde{M})$ is the Banach space $ba(\widetilde{M})$ consisting of all bounded and finitely additive signed measures on $\widetilde{M}$ with the total variation as the norm, that is $\|\mu\|=|\mu|(\widetilde M)$.
Thus, whenever $F\in\Lip_0(M)^\ast$, by the Hahn--Banach extension theorem,
there is a $\mu\in ba(\widetilde{M})$ such that $|\mu|(\widetilde{M})=\|F\|$ and
\begin{equation}\label{eq: F(f)=int_Gamma f dmu for every f in Lip_0(M)}
F(f)=\int_{\widetilde{M}}\widetilde{f}\,d\mu\quad\text{for every $f\in\Lip_0(M)$.}
\end{equation}

We are going to use the following notation. For a subset $A$ of $M$, we write
\[
\Gamma_{1,A}=\{(x,y)\in\widetilde{M}\colon x\in A\}
\quad\text{and}\quad
\Gamma_{2,A}=\{(x,y)\in\widetilde{M}\colon y\in A\}.
\]
%Let us fix some more notation. 
Given a metric space $M$, a point $x$ in $M$, and $r\geq 0$, 
we denote by $B(x,r)$ the open ball in $M$ centred at $x$ of radius $r$. For $x,y\in M$ with $x\neq y$, 
we denote by $m_{x,y}$ the norm-one element $\tfrac{\delta_x-\delta_y}{d(x,y)}$ in $\mathcal{F}(M)$.

%%%%%%%%%%%%%%%%%%%%%%%%%%%%%%%%%%%%%%%%%%%%%%%%%%%%%%%%%%%%%%%%%%%%%%%%%%%%%%%%%%%%%%%%%%%%%%%%%%%%%%%%%%%%%%%%%%
\section{The \texorpdfstring{SSD$2$P}{SSD2P} for spaces of Lipschitz functions}
%\section{Sufficient conditions for the space \texorpdfstring{$\Lip_0(M)$}{Lip(M)} to have the SSD$2$P}
%%%%%%%%%%%%%%%%%%%%%%%%%%%%%%%%%%%%%%%%%%%%%%%%%%%%%%%%%%%%%%%%%%%%%%%%%%%%%%%%%%%%%%%%%%%%%%%%%%%%%%%%%%%%%%%%%%

We start this section by giving two sufficient conditions for the space $\Lip_0(M)$ to have the SSD$2$P. The first one is a consequence of identifying the dual space of $\Lip_0(M)$ via the de Leeuw's transform. From this, we derive the second one, which involves only conditions on the metric of the space $M$ and which we will then use to prove Theorems~\ref{thm: If M is unbounded or not uniformly discrete, then Lip_0(M) has the SSD2P}
and \ref{thm: ruumil Lip_0(K_n) on alati SSD2P}.

%In this section, we give two sufficient conditions for the space $\Lip_0(M)$ to have the SSD$2$P. 
%One of these conditions involves only the underlying metric space $M$. 
%We don't know whether these conditions are also necessary ones. 
% Although both of these conditions are applicable for our purposes, 
%we will make use of the one that involves only the underlying metric space $M$. 
% Namely, we prove that $\Lip_0(M)$ has the SSD$2$P if $M$ satisfies a property we call the \emph{seq-SLTP}. 
%We show that the metric space $M$ has the seq-SLTP whenever $M$ is unbounded, not uniformly discrete, 
%or $M=K_n$, where $n\in\mathbb N$, including the case of $n=2$. 
%We also show that the seq-SLTP is strictly stronger than the SLTP. 
%This leaves open the question of whether the $w^*$-SSD$2$P and the SSD$2$P are equivalent for the spaces of Lipschitz functions. 
%Our first sufficient condition is the following one.

\begin{lemma}\label{lem: main lemma for detecting the SSD2P}
Let $M$ be a pointed metric space
and let $\widetilde{M}$ be as in \eqref{eq: Gamma:=(M times M) setminus ((x,x): x in M)}.
Suppose that, whenever $\delta>0$, $n\in\N$,
$h_1,\dotsc,h_n\in\Lip_0(M)$ with $\|h_i\|\leq 1-\delta$ for every $i\in\{1,\dotsc,n\}$, and
$\mu\in ba(\widetilde{M})$ with only non-negative values,
there exist a subset $A$ of $M$
and functions $f_1,\dotsc,f_n,g\in\Lip_0(M)$ satisfying
\begin{itemize}
\item
$\mu(\Gamma_{1,A})<\delta$ and $\mu(\Gamma_{2,A})<\delta$;
\item
$f_i|_{M\setminus A}=h_i|_{M\setminus A}$ for every $i\in\{1,\dotsc,n\}$;
\item
$g|_{M\setminus A}=0$ and $\|g\|\geq 1-\delta$;
\item
$\|f_i\pm g\|\leq1$ for every $i\in\{1,\dotsc,n\}$.
\end{itemize}
Then the space $\Lip_0(M)$ has the SSD$2$P.
\end{lemma}
\begin{proof}
Let $n\in\N$, let $F_1,\dots,F_n\in S_{\Lip_0(M)^\ast}$, and let $\eps>0$.
It suffices to find $f_i\in S(F_i,\eps)$, $i=1,\dotsc,n$, and $g\in\Lip_0(M)$ with $\|g\|>1-\eps$
such that $f_i\pm g\in S(F_i,\eps)$ for every $i\in\{1,\dotsc,n\}$.

For every $i\in\{1,\dotsc,n\}$, let $\mu_i\in ba(\widetilde{M})$ with $|\mu_i|(\widetilde{M})=1$
satisfy \eqref{eq: F(f)=int_Gamma f dmu for every f in Lip_0(M)}
with $F$ and $\mu$ replaced by $F_i$ and~$\mu_i$, respectively.
Define $\mu = |\mu_1|+\dotsb+|\mu_n|$.
Fix a real number $\delta>0$ satisfying $8\delta\leq\eps$.
For every $i\in\{1,\dotsc,n\}$, pick a function $h_i\in S(F_i, 2\delta)$ with $\|h_i\|\leq 1-\delta$.

Let a subset $A$ of $M$ and functions $f_1,\dotsc,f_n,g\in\Lip_0(M)$ satisfy
the conditions in the lemma.
Setting $\Gamma_A = \Gamma_{1,A}\cup\Gamma_{2,A}$, one has $\mu(\Gamma_A)<2\delta$,
hence, whenever $i\in\{1,\dotsc,n\}$,
\begin{align*}
|F_i(g)|
&=\biggl|\int_{\widetilde{M}}\widetilde{g}\,d\mu_i\biggr|
=\biggl|\int_{\Gamma_A}\widetilde{g}\,d\mu_i\biggr|
\leq|\mu_i|(\Gamma_A)
<2\delta
\end{align*}
and (observing that $\widetilde{f}_i|_{\widetilde{M}\setminus\Gamma_A}=\widetilde{h}_i|_{\widetilde{M}\setminus\Gamma_A}$)
\begin{align*}
F_i(f_i)
&=\int_{\widetilde{M}}\widetilde{f}_i\,d\mu_i
%=\int_{\widetilde{M}\setminus\Gamma_A}\widetilde{f_i}\,d\mu_i+\int_{\Gamma_A}\widetilde{f_i}\,d\mu_i
=\int_{\widetilde{M}}\widetilde{h}_i\,d\mu_i+\int_{\Gamma_A}\bigl(\widetilde{f}_i-\widetilde{h}_i\bigr)\,d\mu_i\\
&=F_i(h_i)+\int_{\Gamma_A}\bigl(\widetilde{f}_i-\widetilde{h}_i\bigr)\,d\mu_i\\
&>1-2\delta-2|\mu_i|(\Gamma_A)
>1-6\delta,
\end{align*}
and thus
\begin{align*}
F_i(f_i\pm g)
&\geq F_i(f_i)-|F_i(g)|>1-6\delta-2\delta\geq 1-\varepsilon.
\end{align*}
\end{proof}

One can prove Theorems \ref{thm: If M is unbounded or not uniformly discrete, then Lip_0(M) has the SSD2P}
and \ref{thm: ruumil Lip_0(K_n) on alati SSD2P} by directly applying Lemma \ref{lem: main lemma for detecting the SSD2P}.
However, we prefer to first prove (and then use) a further sufficient condition for $\Lip_0(M)$ to have the SSD$2$P,
which involves only conditions on the metric of the space $M$ and is therefore easy to handle.

\begin{defn}[cf. {\cite[Definition 1.3]{MR4026495}}]\label{def: seq-SLTP}
We say that a metric space~$M$ has the \emph{sequential strong long trapezoid property} (briefly, \emph {seq-SLTP}) if, 
for every $\varepsilon>0$, there exist pairwise disjoint subsets $A_1, A_2,\dotsc$ of $M$ such that,
for every $m\in\N$, there are $u_m, v_m\in A_m$ with $u_m\neq v_m$ satisfying, for all $x,y \in M\setminus A_m$,
\begin{equation}\label{eq: LTP}
(1-\varepsilon)\bigl(d(x,y)+d(u_m,v_m)\bigr)\leq d(x,u_m)+d(y,v_m),
\end{equation}
and, for all $x,y,z,w \in M\setminus A_m$,
\begin{equation}\label{eq: SLTP}
\begin{aligned}
(1-\varepsilon)&\bigl(d(x,y)+d(z,w)+2d(u_m,v_m)\bigr)\\
&\qquad\qquad\leq d(x,u_m)+d(y,u_m)+d(z,v_m)+d(w,v_m).
\end{aligned}
\end{equation}
\end{defn}

\begin{thm}[cf. {\cite[the proof of Theorem 2.1, (ii)$\Rightarrow$(i)]{MR4026495}}]\label{thm: SSD2P from seq-SLTP}
Let $M$ be a pointed metric space. If $M$ has the seq-SLTP, then $\Lip_0(M)$ has the SSD$2$P.
\end{thm}

\begin{rem}
We do not know whether the converse of Theorem~\ref{thm: SSD2P from seq-SLTP} holds. 
However, the seq-SLTP is strictly stronger than the SLTP (see Example~\ref{ex: SLTP but not seq-SLTP} below).
\end{rem}

\begin{proof}[Proof of Theorem~\ref{thm: SSD2P from seq-SLTP}]
Assume that $M$ has the seq-SLTP. Let $\delta>0$, $n\in \mathbb{N}$, $h_1,\dotsc, h_n\in \Lip_0(M)$ with $\|h_i\|\leq 1-\delta$ for every $i\in \{1,\ldots, n\}$, and let $\mu\in ba(\widetilde{M})$ with only non-negative values where $\widetilde{M}$ is as in $\eqref{eq: Gamma:=(M times M) setminus ((x,x): x in M)}$. By Lemma \ref{lem: main lemma for detecting the SSD2P},
it suffices to find a subset $A$ of $M$ and functions $f_1,\dotsc,f_n,g\in\Lip_0(M)$
satisfying the conditions of that lemma.

By the seq-SLTP, there exist subsets $A_1, A_2, \dots$ of $M$ and points $u_m, v_m\in A_m$, $m=1,2\dots$, as in Definition~\ref{def: seq-SLTP} with $\varepsilon=\delta$. Since the sets $A_1, A_2,\dotsc$ are pairwise disjoint, there exists an $m\in \mathbb{N}$ such that $\mu(\Gamma_{1,A_m})<\delta$ and $\mu(\Gamma_{2,A_m})<\delta$. Let $A = A_m$,  $u = u_m$, and $v=v_m$. In order to define the suitable functions $f_1,\dotsc,f_n,g$, we follow the idea of \cite[proof of Theorem~2.1, (ii)$\Rightarrow$(i)]{MR4026495}.

Setting
\[
r_0=\frac{1}{2}\inf_{x,y\in M\setminus A}\bigl(d(x,u)+d(y,u)-(1-\delta)d(x,y)\bigr)
\]
and
\[
s_0=\frac{1}{2}\inf_{z,w\in M\setminus A}\bigl(d(z,v)+d(w,v)-(1-\delta)d(z,w)\bigr),
\]
one has $r_0+s_0\geq(1-\delta)d(u,v)$. Thus, there exist $r,s\geq 0$ with $r \leq  r_0$ and $s\leq s_0$ such that
\[
r+s=(1-\delta) d(u,v).
\]
We may assume that $r>0$. Define a function $g\colon M\to\mathbb{R}$ by
\[
g(x)=
\begin{cases}
r-d(x,u)&\quad\text{if $x\in B(u,r)$;}\\
-s+d(x,v)&\quad\text{if $x\in B(v,s)$;}\\
0&\quad\text{otherwise.}
\end{cases}
\]
Observe that $\|g\|\leq 1$
(here we use that, whenever $x\in B(u,r)$ and $y\in B(v,s)$, one has $g(y)\leq 0 \leq g(x)$, and thus $|g(x)-g(y)|=g(x)-g(y)$).
One also has $\|g\|\geq 1-\delta$, because
\[
|g(u)-g(v)|=g(u)-g(v)=r+s=(1-\delta) d(u,v).
\]

Set $L=(M\setminus A) \cup B$, where $B=B(u,r)\cup B(v,s)$. Observe that $B\cap (M\setminus A) =\emptyset$. 
Fix $i\in\{1,\dotsc,n\}$. We first define $f_i$ on the set $L$. Let $f_i|_{M\setminus A} = h_i|_{M\setminus A}$. We next show that there is a $c_i\in\mathbb{R}$ such that, 
by
%there is a $c_i\in\mathbb{R}$ such that,
defining $f_i|_{B}=c_i$, one has
%$\|f_i\|_{\Lip_0(L)}\leq 1$, 
$\|f_i\pm g\|_{\Lip_0(L)}\leq 1$ and $\|f_i\pm |g|\|_{\Lip_0(L)}\leq 1$.

Set
\begin{alignat*}{2}
\widecheck{a}_i&=\sup_{x\in M\setminus A}\bigl( h_i(x)-d(x,u)\bigr),\quad
&\widehat{a}_i&=\inf_{x\in M\setminus A}\bigl( h_i(x)+d(x,u)\bigr),\\
\widecheck{b}_i&=\sup_{x\in M\setminus A}\bigl( h_i(x)-d(x,v)\bigr),\quad
&\widehat{b}_i&=\inf_{x\in M\setminus A}\bigl( h_i(x)+d(x,v)\bigr).
\end{alignat*}
Whenever $x,y\in M\setminus A$, since $\| h_i\|\leq 1-\delta$, one has
\[
 h_i(x)+d(x,u)-\bigl( h_i(y)-d(y,u)\bigr)\geq d(x,u)+d(y,u)-(1-\delta) d(x,y)\geq 2r,
\]
and, by \eqref{eq: LTP},
\begin{align*}
 h_i(x)+d(x,u)-\bigl( h_i(y)-d(y,v)\bigr)
&\geq d(x,u)+d(y,v)-(1-\delta)d(x,y)\\
&\geq (1-\delta)d(u,v)> r+s.
\end{align*}
Thus, $\widehat{a}_i-r\geq \widecheck{a}_i+r$ and $\widehat{a}_i-r\geq \widecheck{b}_i+s$.
Similarly, one observes that $\widehat{b}_i-s\geq \widecheck{b}_i+s$
and $\widehat{b}_i-s \geq \widecheck{a}_i+r$.
It follows that there exists a
$c_i\in\bigl[\widecheck{a}_i+r,\widehat{a}_i-r\bigr]\cap\bigl[\widecheck{b}_i+s,\widehat{b}_i-s\bigr]$.
%It is straightforward to verify that 
This $c_i$ does the job. By setting
\[
f_i(y)=\sup_{x\in L}\bigl(f_i(x)+|g(x)|-d(x,y)\bigr)
\quad\text{for every $y\in M\setminus L$,}
\]
one has
%$\|f_i\|_{\Lip_0(L)}\leq 1$, 
$\|f_i\pm g\|_{\Lip_0(M)}\leq 1$. 
For details, we refer the reader to the corresponding parts of \cite[proof of Theorem~2.1, (ii)$\Rightarrow$(i)]{MR4026495}, 
where the argumentation reads nearly word-for-word when replacing $N$ by $M\setminus A$.
\end{proof}

% \begin{thm}[cf. {\cite[Propositions~2.7 and 2.8]{MR4093788}}]\label{thm: ruumil Lip_0(K_n) on alati SSD2P}
% For every $n\in\N$, the space $\Lip_0(K_n)$ has the SSD$2$P.
% \end{thm}

Theorems \ref{thm: If M is unbounded or not uniformly discrete, then Lip_0(M) has the SSD2P}
and \ref{thm: ruumil Lip_0(K_n) on alati SSD2P} are immediate corollaries of
Theorem \ref{thm: SSD2P from seq-SLTP} teamed with the following Propositions
\ref{prop: M unbounded or not uniformly discrete is seq-SLTP} and \ref{prop: MR-l K_n on alati seq-SLTP},
respectively.

\begin{prop}\label{prop: M unbounded or not uniformly discrete is seq-SLTP}
An unbounded or not uniformly discrete metric space has the seq-SLTP.
\end{prop}

\begin{prop}\label{prop: MR-l K_n on alati seq-SLTP}
For every $n\in \N$, the metric space $K_n$ has the seq-SLTP.
\end{prop}

In the proof of Proposition \ref{prop: M unbounded or not uniformly discrete is seq-SLTP},
we make use of the following lemma.

\begin{lemma}\label{lem: LTP ja SLTP ketta jaoks}
Let $\varepsilon>0$, $p\in M$, $0\leq s< r$, and $u,v\in B(p,r)\setminus B(p,s)$ be such that
\[
4s\leq \varepsilon d(u,v)
\]
and, for every $x\in M\setminus B(p,r)$, one has
\[
2d(u,v)\leq \varepsilon\min\bigl\{d(x,u),d(x, v)\bigr\}.
\]
Then, for all $x,y,z,w \in M\setminus A$, where $A = B(p,r)\setminus B(p,s)$, one has
\begin{equation}\label{eq: LTP originaal}
(1-\varepsilon)\bigl(d(x,y)+d(u,v)\bigr)\leq d(x,u)+d(y,v)
\end{equation}
and
\begin{equation}\label{eq: SLTP originaal}
\begin{aligned}
(1-\varepsilon)&\bigl(d(x,y)+d(z,w)+2d(u,v)\bigr)\\
&\qquad\qquad\leq d(x,u)+d(y,u)+d(z,v)+d(w,v).
\end{aligned}
\end{equation}
\end{lemma}
\begin{proof}
Let $x,y,z,w\in M\setminus A$.
First suppose that $x,y,z,w\in B(p,s)$.
For this case, observe that, whenever $a,b,c,d\in B(p,s)$, one has
\[
\eps d(u,v)\geq4s\geq d(a,b)+d(c,d)
\]
and thus
\begin{align*}
d(a,u)+d(b,v)
&\geq
%d(u,v)-d(a,b)=
(1-\eps)d(u,v)+\eps d(u,v)-d(a,b)\\
&\geq(1-\eps)\bigl(d(u,v)+d(c,d)\bigr).
\end{align*}
Taking, in the above, $a=c=x$ and $b=d=y$, one obtains \eqref{eq: LTP originaal}.
Taking, respectively, $a=c=x$, $b=z$, and $d=y$, and $a=y$, $b=d=w$, and $c=z$, one obtains
\[
d(x,u)+d(z,v)\geq(1-\eps)\bigl(d(u,v)+d(x,y)\bigr)
\]
and
\[
d(y,u)+d(w,v)\geq(1-\eps)\bigl(d(u,v)+d(z,w)\bigr),
\]
and \eqref{eq: SLTP originaal} follows.

Now suppose that (at least) one of the points $x$ and $y$, say $x$, is not in $B(0,r)$.
In this case
\begin{align*}
d(x,u)+d(y,v)
&\geq(1-\eps)d(x,u)+2d(u,v)+d(y,v)\\
&\geq(1-\eps)\bigl(d(x,u)+d(u,v)+d(y,v)+d(u,v)\bigr)\\
&\geq(1-\eps)\bigl(d(x,y)+d(u,v)\bigr).
\end{align*}
Finally, if (at least) one of the points  $x,y,z,w$, say, again, $x$, is not in $B(0,r)$, then
\begin{multline*}
d(x,u)+d(y,u)+d(z,v)+d(w,v)\\
\begin{aligned}
&\geq(1-\eps)\bigl(d(x,u)+d(y,u)+d(z,v)+d(w,v)\bigr)+2d(u,v)\\
&\geq(1-\eps)\bigl(d(x,y)+d(z,w)+2d(u,v)\bigr),
\end{aligned}
\end{multline*}
and the proof is complete.
\end{proof}

\begin{proof}[Proof of Proposition {\ref{prop: M unbounded or not uniformly discrete is seq-SLTP}}]
%We make use of Theorem~\ref{thm: SSD2P by M}.
Let $M$ be a pointed metric space, and let $\eps>0$.

First assume that $M$ is unbounded.
Letting $r_0=1$, we can inductively define points $u_m,v_m\in M\setminus B(0,r_{m-1})$ with $u_m\neq v_m$
and real numbers $r_m>0$, $m=1,2,\dotsc$,
satisfying, for every $m\in\N$, the inequalities $4r_{m-1}\leq\varepsilon d(u_m,v_m)$, $r_m>r_{m-1}$, and
\begin{equation}\label{eq: 2d(u_m,v_m)=< ... for every x in M setminus B(0,r_m)}
2d(u_m,v_m)\leq \varepsilon\min\bigl\{d(x,u_m),d(x, v_m)\bigr\}
\quad\text{for every $x\in M\setminus B(0,r_m)$.}
\end{equation}
Now the sets $A_m\coloneq B(0, r_{m})\setminus B(0, r_{m-1})$, $m=1,2,\dotsc$, are pairwise disjoint.
For every $m\in\N$, Lemma \ref{lem: LTP ja SLTP ketta jaoks}
with $p=0$, $r=r_{m}$, $s=r_{m-1}$, $u=u_m$, and $v=v_m$
implies that, for all $x,y,z,w\in M\setminus A_m$,
the inequalities (\ref{eq: LTP}) and (\ref{eq: SLTP}) hold.

Assume now that $M$ is not uniformly discrete.
We first consider the case when $M$ has a limit point $p$.
Starting with $r_1=1$, we can inductively define points $u_m,v_m\in B(p,r_{m})\setminus\{p\}$ with $u_m\neq v_m$
and real numbers $r_m>0$, $m=1,2,\dotsc$,
satisfying, for every $m\in\N$, the condition \eqref{eq: 2d(u_m,v_m)=< ... for every x in M setminus B(0,r_m)},
$r_{m+1}<r_m$, and $4r_{m+1}\leq\varepsilon d(u_m,v_m)$.
Now the sets $A_m\coloneq B(p, r_{m})\setminus B(p, r_{m+1})$, $m=1,2,\dotsc$, are pairwise disjoint.
For every $m\in\N$, Lemma \ref{lem: LTP ja SLTP ketta jaoks} with  $r=r_{m}$, $s=r_{m+1}$, $u=u_m$, and $v=v_m$
implies that, for all $x,y,z,w\in M\setminus A_m$,
the inequalities (\ref{eq: LTP}) and (\ref{eq: SLTP}) hold.

Finally, consider the case when $M$ has no limit points.
Since $M$ is not uniformly discrete, there exist points $u_m,v_m\in M$ with $u_m\neq v_m$, $m=1,2,\dotsc$,
such that $d(u_m,v_m)\to 0$.
Since $M$ has no limit points, we may assume, after passing to subsequences if necessary,
that there is an $r>0$ such that $d(u_m,u_n)\geq2r$ whenever $m,n\in\N$ with $m\neq n$.
Furthermore, we may assume that $v_m\in B(u_m,r)$ and $4d(u_m,v_m)\leq\eps r$ for every $m\in\N$.
Now the sets $A_m\coloneq B(u_m,r)$, $m=1,2,\dotsc$, are pairwise disjoint.
For every $m\in\N$ and every $x\in M\setminus B(u_m,r)$, one has
\[
\eps d(x,u_m)\geq\eps r>2d(u_m,v_m)
\]
and
\[
\eps d(x,v_m)
\geq\eps\bigl(d(x,u_m)-d(u_m,v_m)\bigr)
>\eps\bigl(r-\tfrac{1}{2}r\bigr)
=\tfrac{1}{2}\eps r
\geq2d(u_m,v_m),
\]
thus
Lemma \ref{lem: LTP ja SLTP ketta jaoks} with $p=u=u_m$, $s=0$, and $v=v_m$
implies that, for all $x,y,z,w\in M\setminus A_m$,
the inequalities (\ref{eq: LTP}) and (\ref{eq: SLTP}) hold.
\end{proof}

\begin{proof}[Proof of Proposition \ref{prop: MR-l K_n on alati seq-SLTP}]
%We make use of Theorem~\ref{thm: SSD2P by M}. 
Let $n\in \N$. For every $m\in \N$, define
\begin{align*}
A_m = \big\{(x_j)_{j=1}^\infty \in K_n\colon \text{$\max_{j} x_{j} = n$ and $x_j<n$ for every $j \not\in\{ 2m-1, 2m\}$} \big\},
\end{align*}
$u_m = ne_{2m-1}+(n-1)e_{2m}$, and $v_m = (n-1)e_{2m-1}+ne_{2m}$. 
Note that the sets $A_1, A_2,\dotsc$ are pairwise disjoint. 
Fix an $m\in \N$. Clearly, $u_m, v_m\in A_m$ and $d(u_m, v_m)=1$. 
We show that, for all $x,y\in K_n\setminus A_m$,
\begin{equation*}\label{eq: LTP_wo_eps}
d(x,y)+d(u_m,v_m)\leq d(x,u_m)+d(y,v_m),
\end{equation*}
and, for all $x,y,z,w\in K_n\setminus A_m$,
\begin{equation*}\label{eq: SLTP_wo_eps}
\begin{aligned}
&d(x,y)+d(z,w)+2d(u_m,v_m)\\
&\qquad\qquad\leq d(x,u_m)+d(y,u_m)+d(z,v_m)+d(w,v_m).
\end{aligned}
\end{equation*}
Fix $x = (x_j)_{j=1}^\infty,y = (y_j)_{j=1}^\infty\in K_n\setminus A_m$.  
It suffices to show that the following inequalities hold:
\begin{align*}
d(x,y)+d(u_m, v_m)&\leq d(x, u_m) + d(y, v_m),\\
d(x,y)+d(u_m, v_m)&\leq d(x, u_m) + d(y, u_m),\\
d(x,y)+d(u_m, v_m)&\leq d(x, v_m) + d(y, v_m).
\end{align*}
To this end, let $j\in \N$ be such that $d(x,y)=|x_j-y_j|$. Without loss of generality, we assume that $x_j\geq y_j$.
Notice that, if $j \not\in \{2m-1, 2m\}$, then $x_j\geq d(x,y)$, and therefore $d(x, u_m)\geq d(x,y)$ and $d(x, v_m)\geq d(x,y)$. 
If $j\in \{2m-1, 2m\}$, then $y_j\leq n-1-d(x,y)$ because $x_j\leq n-1$, and hence $d(y, u_m)\geq d(x,y)$ and $d(y, v_m)\geq d(x,y)$. 
Since $d(u_m, v_m)= 1$, the desired inequalities hold.
\end{proof}

%\newpage

Recall \cite[Definition 1.3]{MR4026495} that a metric space~$M$ has the \emph{strong long trapezoid property} (briefly, \emph {SLTP}\label{def: SLTP}) if, 
for every $\varepsilon>0$ and every finite subset $N$ of $M$, there exist elements $u,v\in M$ with $u\neq v$ satisfying, for all $x,y \in N$,
\begin{equation*}
(1-\varepsilon)\bigl(d(x,y)+d(u,v)\bigr)\leq d(x,u)+d(y,v),
\end{equation*}
and, for all $x,y,z,w \in N$,
\begin{equation*}
\begin{aligned}
(1-\varepsilon)&\bigl(d(x,y)+d(z,w)+2d(u,v)\bigr)\\
&\qquad\qquad\leq d(x,u)+d(y,u)+d(z,v)+d(w,v).
\end{aligned}
\end{equation*}

We end this section by giving an example of a metric space $M$ with the SLTP but without the seq-SLTP. 
In fact, this $M$ does not even have the seq-LTP (see Definition~\ref{def: seq-LTP} below). 
By \cite[Theorem 2.1]{MR4233633}, the space $\Lip_0(M)$ has the $w^*$-SSD$2$P. 
It remains unknown whether $\Lip_0(M)$ has the SSD$2$P.

\begin{ex}\label{ex: SLTP but not seq-SLTP}
Let $M = \{a_k, b_k, c_k\colon k\in \N\}$ be the metric space where, for every $k\in \N$, 
\[
d(a_k, c_k)= 2,
\]
and, for all $k,l\in \N$ with $k<l$, 
\[
d(a_k, b_l) = d(b_k, b_l) = d(c_k, b_l)=2,
\]
and the distance between two different elements is $1$ in all other cases.

We first show that the space $M$ has the SLTP. 
To this end, let $N$ be a finite subset of $M$. 
Then there exists a $K\in\N$ such that $N\subset\{a_k, b_k, c_k\colon k<K\}$. 
Let $u=b_K$ and $v = b_{K+1}$. Then, for every $x\in N$, one has $d(x,u)=2$ and $d(x,v) = 2$, and therefore, for all $x,y\in N$,
\[
d(x,y)+d(u,v) \leq 4 = d(x,u)+d(y,v),
\]
and, for all $x,y,z,w\in N$,
\[
d(x,y)+d(z,w)+2d(u,v)\leq 8 = d(x,u)+d(y,u)+d(z,v)+d(w,v).
\]

Now suppose for contradiction that $M$ has the seq-SLTP.
Then there exist pairwise disjoint subsets $A_1$, $A_2$, and $A_3$ of $M$
such that, for every $m\in\{1,2,3\}$, there are $u_m,v_m\in A_m$ with $u_m\neq v_m$
such that the inequality \eqref{eq: LTP} with $\eps< 1/3$ holds for all $x,y\in M\setminus A_m$.
Let $K\in \N$ be such that $u_1,v_1,u_2,v_2,u_3,v_3\in \{a_k, b_k, c_k\colon k<K\}$. 
For every $m\in\{1,2,3\}$, one has
\[
\bigl(1-\varepsilon\bigr)\bigl(d(a_K,c_K)+d(u_m,v_m)\bigr)
\geq 3(1-\varepsilon) > 2= d(a_K,u_m)+d(c_K,v_m),
\]
which implies $a_K\in A_m$ or $c_K\in A_m$.
It follows that $A_1$, $A_2$, and $A_3$ are not pairwise disjoint, a contradiction.
\end{ex}

\section%{The SD2P for spaces of Lipschitz functions}
{The \texorpdfstring{SSD$2$P}{SSD2P}, \texorpdfstring{SD$2$P}{SD2P}, and \texorpdfstring{D$2$P}{D2P} are three different properties for spaces of Lipschitz functions}

In this section, we give an example of a metric space $M$ 
such that the corresponding space $\Lip_0(M)$ has the SD$2$P but fails the SSD$2$P, 
and of a metric space~$M$ such that the corresponding space $\Lip_0(M)$ has the D$2$P but fails the SD$2$P, 
thus showing that the SSD$2$P, SD$2$P, and D$2$P are three different properties for the spaces of Lipschitz functions.
This answers an implicit question in \cite[Introduction]{MR4093788}.
For these two examples, we first give sufficient conditions for the space of Lipschitz functions to have the SD$2$P, and the D$2$P, as we did for the SSD$2$P in Section~2.
We start with an analogue of Lemma \ref{lem: main lemma for detecting the SSD2P} for the SD$2$P.

% Our first aim in this section is to give an example of a metric space $M$ 
% such that the corresponding space $\Lip_0(M)$ has the SD$2$P but fails the SSD$2$P
% (see Example~\ref{ex: seq-LTP but not SLTP}).
% For this example, we need an analogue of Theorem \ref{thm: SSD2P from seq-SLTP} for the SD$2$P.
% For such an analogue, in turn, we need the following analogue of Lemma \ref{lem: main lemma for detecting the SSD2P} for the SD$2$P.

%In this section, we give an example of a metric space $M$ 
%such that the corresponding space $\Lip_0(M)$ has the SD$2$P but fails the SSD$2$P
%(see Example \ref{ex: seq-LTP but not SLTP}), 
%and of a metric space~$M$ such the corresponding space $\Lip_0(M)$ has the D$2$P but fails the SD$2$P
%(see Example~\ref{ex: D2P but not LTP}), 
%thus showing that the SSD$2$P, SD$2$P, and D$2$P are three different properties for the space $\Lip_0(M)$.
%For the first of these examples, we need an analogue of Theorem \ref{thm: SSD2P from seq-SLTP} for the SD$2$P.
%For such an analogue, in turn, we need the following analogue of Lemma \ref{lem: main lemma for detecting the SSD2P} for the SD$2$P.

%In this section, we give two sufficient conditions for the space $\Lip_0(M)$ to have the SD$2$P. 
%We don't know whether these conditions are also necessary ones. 
%Our first sufficient condition is the following one.

\begin{lemma}\label{lem: main lemma for detecting the SD2P}
Let $M$ be a pointed metric space
and let $\widetilde{M}$ be as in \eqref{eq: Gamma:=(M times M) setminus ((x,x): x in M)}.
Suppose that, whenever $\delta>0$, $n\in\N$, $h_1,\dotsc,h_n\in\Lip_0(M)$ with $\|h_i\|\leq 1-\delta$ for every $i\in\{1,\dotsc,n\}$, 
and $\mu\in ba(\widetilde{M})$ with only non-negative values, 
there exist a subset $A$ of $M$, elements $u,v\in A$ with $u\neq v$,
and functions $f_1,\dotsc,f_n\in B_{\Lip_0(M)}$ satisfying
\begin{itemize}
\item
$\mu(\Gamma_{1,A})<\delta$ and $\mu(\Gamma_{2,A})<\delta$;
\item
$f_i|_{M\setminus A}=h_i|_{M\setminus A}$ for every $i\in\{1,\dotsc,n\}$;
\item
$f_i(u)-f_i(v)\geq (1-\delta)d(u,v)$ for every $i\in\{1,\dotsc,n\}$.
\end{itemize}
Then the space $\Lip_0(M)$ has the SD$2$P.
\end{lemma}
\begin{proof}
Let $n\in\N$, let $F_1,\dots,F_n\in S_{\Lip_0(M)^\ast}$, and let $\eps>0$. 
By \cite[Corollary 2.2]{MR3150166} and \cite[Proposition 2.2]{MR3346197}, 
it suffices to find $u,v\in M$  such that $\|F_i + m_{u,v}\|\geq 2-\varepsilon$ for every $i\in\{1,\dotsc,n\}$.

For every $i\in\{1,\dotsc,n\}$, let $\mu_i\in ba(\widetilde{M})$ with $|\mu_i|(\widetilde{M})=1$
satisfy \eqref{eq: F(f)=int_Gamma f dmu for every f in Lip_0(M)}
with~$F$ and $\mu$ replaced by $F_i$ and~$\mu_i$, respectively.
Define $\mu=|\mu_1|+\dotsb+|\mu_n|$.
Fix a real number $\delta>0$ satisfying $7\delta\leq\varepsilon$.
For every $i\in\{1,\dotsc,n\}$, pick a function $h_i\in S(F_i,2\delta)$ with $\|h_i\|\leq 1-\delta$.

Let a subset $A$ of $M$, elements $u,v\in A$ with $u\neq v$, and functions $f_1,\dotsc,f_n\in\Lip_0(M)$ satisfy
the conditions in the lemma. Setting $\Gamma_A = \Gamma_{1,A}\cup\Gamma_{2,A}$, one has $\mu(\Gamma_A)<2\delta$,
hence, whenever $i\in\{1,\dotsc,n\}$, (observing that $\widetilde{f}_i|_{\widetilde{M}\setminus\Gamma_A}=\widetilde{h}_i|_{\widetilde{M}\setminus\Gamma_A}$)
\begin{align*}
F_i(f_i)
&=\int_{\widetilde{M}}\widetilde{f}_i\,d\mu_i
%=\int_{\widetilde{M}\setminus\Gamma_A}\widetilde{f_i}\,d\mu_i+\int_{\Gamma_A}\widetilde{f_i}\,d\mu_i
=\int_{\widetilde{M}}\widetilde{h}_i\,d\mu_i+\int_{\Gamma_A}\bigl(\widetilde{f}_i-\widetilde{h}_i\bigr)\,d\mu_i\\
&=F_i(h_i)+\int_{\Gamma_A}\bigl(\widetilde{f}_i-\widetilde{h}_i\bigr)\,d\mu_i\\
&>1-2\delta-2|\mu_i|(\Gamma_A)
>1-6\delta,
\end{align*}
and thus
\begin{align*}
(F_i+ m_{u,v})(f_i)
&= F_i(f_i)+\frac{f_i(u)-f_i(v)}{d(u,v)}>1-6\delta+1-\delta\geq 2-\varepsilon.
\end{align*}
\end{proof}

We next prove (and then use) a further sufficient condition for $\Lip_0(M)$ to have the SD$2$P%
---an analogue of Theorem \ref{thm: SSD2P from seq-SLTP}---%
which involves only conditions on the metric of the space $M$ and is therefore easy to handle.

\begin{defn}[cf. {\cite[Theorem 3.1, (3)]{MR3803112}}]\label{def: seq-LTP}
We say that a metric space $M$ has the \emph{sequential long trapezoid property} (briefly, \emph{seq-LTP}) if, for every $\varepsilon>0$, there exist pairwise disjoint subsets $A_1, A_2,\dotsc$ of $M$ such that,
for every $m\in\N$, there are $u_m, v_m\in A_m$ with $u_m\neq v_m$ satisfying, for all $x,y \in M\setminus A_m$,
\begin{equation}\label{eq: seq-LTP}
(1-\varepsilon)\bigl(d(x,y)+d(u_m,v_m)\bigr)\leq d(x,u_m)+d(y,v_m).
\end{equation}
\end{defn}
\noindent
Clearly every metric space with the seq-SLTP has the seq-LTP as the conditions (\ref{eq: LTP}) and (\ref{eq: seq-LTP}) are the same.

\begin{thm}[cf. {\cite[Theorem 3.1, (3)$\Rightarrow$(1)]{MR4026495}}]\label{thm: SD2P from seq-LTP}
Let $M$ be a pointed metric space. If $M$ has the seq-LTP, then $\Lip_0(M)$ has the SD$2$P.
\end{thm}

\begin{rem}
We do not know whether the converse of Theorem~\ref{thm: SD2P from seq-LTP} holds. 
Note that the seq-LTP is strictly stronger than the LTP (see Example~\ref{ex: SLTP but not seq-SLTP} above).
\end{rem}

\begin{proof}[Proof of Theorem~\ref{thm: SD2P from seq-LTP}]
Assume that $M$ has the seq-LTP. Let $\delta>0$, let $n\in \mathbb{N}$, let $h_1,\dotsc, h_n\in \Lip_0(M)$ with $\|h_i\|\leq 1-\delta$ for every $i\in \{1,\ldots, n\}$, and let $\mu\in ba(\widetilde{M})$ with only non-negative values where $\widetilde{M}$ is as in $\eqref{eq: Gamma:=(M times M) setminus ((x,x): x in M)}$. By Lemma \ref{lem: main lemma for detecting the SD2P},
it suffices to find a subset $A$ of $M$, elements $u,v\in A$ with $u\neq v$, and functions $f_1,\dotsc,f_n\in\Lip_0(M)$
satisfying the conditions of that lemma. 

By the seq-LTP, there exist subsets $A_1, A_2, \dots$ of $M$ and points $u_m, v_m\in A_m$, $m=1,2,\dots$, as in Definition~\ref{def: seq-LTP} with $\varepsilon=\delta$. Since the sets $A_1, A_2,\dotsc$ are pairwise disjoint, there exists an $m\in \mathbb{N}$ such that $\mu(\Gamma_{1,A_m})<\delta$ and $\mu(\Gamma_{2,A_m})<\delta$. Let $A = A_m$,  $u = u_m$, and $v=v_m$.  

Fix $i\in \{1,\dotsc,n\}$. Define the function $f_i$ by $f_i|_{M\setminus A} = h_i|_{M\setminus A}$,
\[
    f_i(u) = \inf_{x\in M\setminus A}\big(f_i(x)+d(x,u)\big),
\]
and
\[
    f_i(y) = \sup_{x\in \{u\}\cup M\setminus A}\big(f_i(x)-d(x,y)\big) \quad\text{for every $y\in A\setminus \{u\}$}.
\]
Since $\|f_i\|\leq 1$, it remains to show that $f_i(u)-f_i(v)\geq (1-\delta)d(u,v)$. If $f_i(v) = f_i(u)-d(u,v)$, then the inequality holds. Suppose now that this is not the case. Then
\begin{align*}
f_i(u)-f_i(v) &= \inf_{x,y\in M\setminus A}\big(f_i(x)+d(x,u)-f_i(y)+d(y,v)\big)\\
    &\geq \inf_{x,y\in M\setminus A}\big(-(1-\delta)d(x,y)+ d(x,u)+d(y,v)\big)\\
    &\geq (1-\delta) d(u,v).
\end{align*}
\end{proof}

The following example from \cite[Example 3.1]{MR4026495} was the first known example of a metric space $M$ which has the LTP but not the SLTP or,
equivalently, for which the corresponding Lipschitz function space $\Lip_0(M)$ has the $w^*$-SD$2$P but not the $w^*$-SSD$2$P. 
We further show that $M$ has the seq-LTP. Therefore, $\Lip_0(M)$ has the SD$2$P but not the $w^*$-SSD$2$P. 
To our knowledge, this is the first known example of such a Lipschitz function space, 
showing that the properties SD$2$P and ($w^*$-)SSD$2$P are really different for the class of Lipschitz function spaces.

\begin{ex}\label{ex: seq-LTP but not SLTP}
Let $M=\{a_1, a_2, b_1, b_2\}\cup \{u_m,v_m\colon m\in \mathbb{N}\}$ be the metric space where, 
for all $i,j\in \{1,2\}$ and all $m\in \mathbb{N}$,
\begin{align*}
d(a_i,b_j)=d(a_i,u_m)=d(b_i,v_m)=d(u_m,v_m)=1,
\end{align*}
and the distance between two different elements is $2$ in all other cases.

We show that $M$ has the seq-LTP. 
To this end, it suffices to show that, whenever $m\in\N$, one has, for all $x,y\in M\setminus A_m$,
\begin{equation*}
d(x,y)+d(u_m,v_m)\leq d(x,u_m)+d(y,v_m),
\end{equation*}
where $A_m=\{u_m,v_m\}$. Fix an $m\in \N$ and let $x,y\not\in A_m$. 
If $d(x,u_m)+d(y,v_m)\geq 3$, then the desired inequality holds because $d(u_m,v_m)=1$. 
If $d(x,u_m)+d(y,v_m) = 2$, then $x\in \{a_1,a_2\}$ and $y\in \{b_1,b_2\}$, and therefore $d(x,y) = 1$;
thus the desired inequality holds.
\end{ex}

%\section{The D$2$P for spaces of Lipschitz functions\\
%Sufficient conditions for the space \texorpdfstring{$\Lip_0(M)$}{Lip(M)} to have the D$2$P}

Our second aim in this section is to give an example of a metric space $M$ 
such the corresponding space $\Lip_0(M)$ has the D$2$P but fails the SD$2$P.
For this example, we need  a sufficient condition for the space $\Lip_0(M)$ to have the D$2$P.
Here we do not have a condition which involves only the metric of the underlying space $M$ 
as we did with Theorems \ref{thm: SSD2P from seq-SLTP} and \ref{thm: SD2P from seq-LTP} for the SSD$2$P and the SD$2$P, respectively.
However, the following analogue of Lemmas \ref{lem: main lemma for detecting the SSD2P}
and \ref{lem: main lemma for detecting the SD2P} for the D$2$P is suitable for our purposes.
 
%In this section, we give a sufficient condition for the space $\Lip_0(M)$ to have the D$2$P. 
% We don't know whether this condition is also a necessary one. Here we don't have a sufficient condition which involves only the underlying metric %space $M$ as we had in the previous two sections for the SSD2P and the SD2P. However, we are able to use this condition directly to show that the %($w^*$-)SD2P and the D2P are different for the spaces of Lipschitz functions. It remains open whether the D2P and the slice-D2P are different for %the spaces of Lipschitz functions.

\begin{lemma}\label{lem: main lemma for detecting the D2P}
Let $M$ be a pointed metric space
and let $\widetilde{M}$ be as in \eqref{eq: Gamma:=(M times M) setminus ((x,x): x in M)}.
Suppose that, whenever $\delta>0$, $h\in\Lip_0(M)$ with $\|h\|\leq 1-\delta$, and $\mu\in ba(\widetilde{M})$ with only non-negative values,
there exist a subset $A$ of $M$, elements $u,v\in A$ with $u\neq v$,
and functions $f, g \in B_{\Lip_0(M)}$ satisfying
\begin{itemize}
\item
$\mu(\Gamma_{1,A})<\delta$ and $\mu(\Gamma_{2,A})<\delta$;
\item
$f|_{M\setminus A}=g|_{M\setminus A} = h|_{M\setminus A}$;
\item
$f(u)-f(v)\geq (1-\delta)d(u,v)$ and $g(u)-g(v)\leq -(1-\delta)d(u,v)$.
\end{itemize}
Then the space $\Lip_0(M)$ has the D$2$P.
\end{lemma}

%\begin{rem}
%
%\end{rem}

%
\begin{proof}[Proof of Lemma \ref{lem: main lemma for detecting the D2P}]
Let $n\in\N$, let $F_1,\dotsc, F_n\in S_{\Lip_0(M)^\ast}$, let $\varepsilon>0$, and let $\phi\in B_{\Lip_0(M)}$.
It suffices to find $f,g\in \Lip_0(M)$ with $\|f\|\leq 1$ and $\|g\| \leq 1$ such that $|F_i(f-\phi)|<\varepsilon$ and $|F_i(g-\phi)|<\varepsilon$ for every $i\in \{1,\dotsc, n\}$, and $\|f-g\|> 2-\varepsilon$. 

For every $i\in\{1,\dotsc,n\}$, let $\mu_i\in ba(\widetilde{M})$ with $|\mu_i|(\widetilde{M})=1$
satisfy \eqref{eq: F(f)=int_Gamma f dmu for every f in Lip_0(M)}
with $F$ and $\mu$ replaced by $F_i$ and~$\mu_i$, respectively.
Define $\mu=|\mu_1|+\dotsb+|\mu_n|$.
Fix a real number $\delta>0$ satisfying $5\delta\leq\varepsilon$. Let $h=(1-\delta)\phi$.

Let a subset $A$ of $M$, elements $u,v\in A$ with $u\neq v$, and functions $f,g\in\Lip_0(M)$ satisfy the conditions in the lemma.
Setting $\Gamma_A=\Gamma_{1,A}\cup\Gamma_{2,A}$, one has $\mu(\Gamma_A)<2\delta$,
hence, for every $i\in \{1,\dotsc, n\}$ (observing that $\widetilde{f}|_{\widetilde{M}\setminus\Gamma_A}=\,\widetilde{h}|_{\widetilde{M}\setminus\Gamma_A}$),
\begin{align*}
|F_i(f-\phi)|&\leq|F_i(f-h)|+\delta
=\bigg|\int_{\widetilde{M}}\bigl(\widetilde{f}-\widetilde{h}\bigr)\,d\mu_i\bigg|+\delta
%=\int_{\widetilde{M}\setminus\Gamma_A}\widetilde{f_i}\,d\mu_i+\int_{\Gamma_A}\widetilde{f_i}\,d\mu_i
=\bigg|\int_{\Gamma_A}\bigl(\widetilde{f}-\widetilde{h}\bigr)\,d\mu_i\bigg|+\delta\\
&\leq 2|\mu_i|(\Gamma_A)+\delta < 5\delta \leq \varepsilon.
\end{align*}
Similarly, $|F_i(g-\phi)|< \varepsilon$ for every $i\in\{1,\dotsc,n\}$. It remains to observe that
\begin{align*}
    \|f-g\|&\geq \frac{(f-g)(u)-(f-g)(v)}{d(u,v)} = \frac{f(u)-f(v)-g(u)+g(v)}{d(u,v)}\\
    &\geq \frac{2(1-\delta)\,d(u,v)}{d(u,v)} = 2(1-\delta)>2-\varepsilon.
\end{align*}
\end{proof}

We now introduce a space $M$ for which the corresponding Lipschitz function space $\Lip_0(M)$ has the D$2$P but not the ($w^*$-)SD$2$P. 
To our knowledge this is the first such example in the class of Lipschitz function spaces.

Recall \cite[Theorem 3.1, (3)]{MR3803112} that  a metric space~$M$ has the \emph{long trapezoid property} (briefly, \emph {LTP}\label{def: LTP}) if, 
for every $\varepsilon>0$ and every finite subset $N$ of $M$, there exist elements $u,v\in M$ with $u\neq v$ satisfying, for all $x,y \in N$,
\begin{equation*}
(1-\varepsilon)\bigl(d(x,y)+d(u,v)\bigr)\leq d(x,u)+d(y,v).
\end{equation*}

\begin{ex}\label{ex: D2P but not LTP}
Let $M = \big\{a_i, u^i_m, v^i_m \colon i\in \{1,2,3\}, m\in \N\big\}$ be the metric space 
where, for all $i,j\in\{1,2,3\}$ with $i\neq j$ and all $m\in \N$, 
\[
    d(a_i, u_m^j)= d(a_i, v_m^j) = 1,
\]
and, for all $j\in\{1,2,3\}$ and all $m\in\N$,
\[
d(u_m^j, v_m^j) = 1,
\]
and the distance between two different elements is $2$ in all other cases.

We first show that the space $\Lip_0(M)$ has the D$2$P. We make use of Lemma~\ref{lem: main lemma for detecting the D2P}. 
Let $\delta>0$, let $h\in \Lip_0(M)$ with $\|h\|\leq 1$, and let $\mu\in ba(\widetilde{M})$ with only non-negative values. 
We may assume that $h(a_1)\leq h(a_2)\leq h(a_3)$. 
%Then $h(a_2)-h(a_1)\leq 1$ or $h(a_3)-h(a_2)\leq 1$. 
Set $L=\inf h(M)$.
If $h(a_2)\leq L+1$, then let $k=3$ and $c = L$; 
otherwise, let $k=1$ and $c=L+1$.
Choose an $m\in \N$ so that $\mu(\Gamma_{1,A})<\delta$ and $\mu(\Gamma_{2,A})<\delta$ where  $A=\{u_m^k, v_m^k\}$. 
Let $u = u_m^k$ and $v = v_m^k$, and define $f, g\colon M\to \R$ by
\begin{align*}
    f(x) = 
    \begin{cases}
    h(x) &\quad\text{if $x\in M\setminus A$};\\
    c+1 &\quad\text{if $x = u$};\\
    c &\quad\text{if $x = v$},
    \end{cases}
    \qquad \text{and}\qquad
    g(x) = 
    \begin{cases}
    h(x) &\quad\text{if $x\in M\setminus A$};\\
    c &\quad\text{if $x = u$};\\
    c+1 &\quad\text{if $x = v$}.
    \end{cases}
\end{align*}
Then 
\[
f(u) -f(v) = g(v)-g(u)= 1 =d(u,v).
\]
It is straightforward to verify that $\|f\| = 1$ and $\|g\| = 1$. 
By Lemma \ref{lem: main lemma for detecting the D2P}, $\Lip_0(M)$ has the D$2$P.

We now show that the space $\Lip_0(M)$ does not have the $w^*$-SD$2$P. 
It suffices to show that space $M$ does not have the LTP. 
Let $N = \{a_1,a_2,a_3\}$ and let $\varepsilon < 1/3$. 
Whenever $u,v\in M$ with $u\neq v$, there exist $x,y\in N$ with $x\neq y$ such that $d(x,u)\leq1$ and $d(y,v)\leq1$. 
Since $d(x,y)=2$, one has
\[
(1-\varepsilon)\bigl(d(u,v)+d(x,y)\bigr)\geq 3(1-\varepsilon)>2 \geq d(x,u)+d(y,v).
\]
\end{ex}

\section{Local norm-one Lipschitz function is a Daugavet point}

Let $M$ be a pointed metric space. In this section, we show that certain norm-one elements $f$ of $\Lip_0(M)$ are \emph{Daugavet points}, 
i.e, given a slice $S$ of the unit ball of $\Lip_0(M)$ and an $\varepsilon>0$, there exists a $g\in S$ with $\|f- g\| > 2 - \varepsilon$.

\begin{defn}[see {\cite[Definition~2.5]{MJAR}}]
A function $f\in \Lip_0(M)$ is said to be \emph{local} if, for every $\varepsilon>0$, there are $u,v\in M$ with $u\neq v$ such that $d(u,v)<\varepsilon$ and $f(m_{u,v})>\|f\|-\varepsilon$.
\end{defn}

The question of whether every local $f$ in the unit sphere of $\Lip_0(M)$ is a Daugavet point was posed and addressed in \cite{MJAR}; there it was shown that
\begin{enumerate}
\item every local $f$ on the unit sphere of $\Lip_0(M)$ is a weak$^*$ Daugavet point, i.e., given a weak$^*$ slice $S$ of the unit ball of $\Lip_0(M)$ and an $\varepsilon>0$, there exists $g\in S$ with $\|f - g\| > 2- \varepsilon$;
\item every spreadingly local $f$ (see \cite[Definition~2.5]{MJAR}) in the unit sphere of $\Lip_0(M)$ is a Daugavet point.
\end{enumerate}
The result (2) is yielded as a local argument of \cite[Theorem 3.1]{MR2379289} 
which says that if $M$ is spreadingly local then $\Lip_0(M)$ has the Daugavet property. 
The way we look at $\Lip_0(M)^*$ allows us to improve upon all these results and fully answer the aforementioned question with Theorem~\ref{thm: If f in Lip_0(M) is local, then f is a Daugavet point}.
The latter is in fact a straightforward consequence of the following result.

\begin{prop}[cf. {\cite[Theorem~2.6]{MJAR}}]\label{prop: d(u_n,v_n)->0}
Let $M$ be a pointed metric space, let $F\in S_{\Lip_0(M)^*}$, 
and let $(u_n)$ and $(v_n)$ be two sequences of elements in $M$ such that $u_n\neq v_n$ for every $n\in\mathbb N$. 
If $d(u_n,v_n)\to 0$, then $\|F+m_{u_n,v_n}\|\to 2$.
\end{prop}

In the proof of Proposition \ref{prop: d(u_n,v_n)->0},
we shall repeatedly make use of the following simple lemma.

\begin{lemma}\label{lem: if y in B(u,theta r) and x in M setminus B(u,r)}
Let $M$ be a metric space, let $u,x,y\in M$, and let $r$ and $\theta$ be real numbers with $r>0$ and $0<\theta<1$.
Suppose that $y\in B(u,\theta r)$ and $x\in M\setminus B(u,r)$.
Then $d(y,u)\leq \frac{\theta}{1-\theta}\,d(x,y)$.
\end{lemma}
\begin{proof}
Observe that $d(x,y)\geq(1-\theta)r$
because otherwise one would have
\[
d(x,u)\leq d(x,y)+d(y,u)<(1-\theta)r+\theta r=r,
\]
a contradiction.
It follows that $d(y,u)<\theta r\leq\frac{\theta}{1-\theta}\,d(x,y)$, as desired.
\end{proof}

\begin{proof}[Proof of Proposition~\ref{prop: d(u_n,v_n)->0}]
Assume that $d(u_n, v_n)\to 0$. Let $\varepsilon>0$. Our aim is to prove that there exists an $n\in \N$ such that $\|F+m_{u_n,v_n}\|> 2-\varepsilon$. To that end, fix a real number $\delta>0$ satisfying $7\delta\leq \varepsilon$ and $h\in S(F, 2\delta)$ with $\|h\|\leq 1-\delta$. Let $\widetilde{M}$ be as in \eqref{eq: Gamma:=(M times M) setminus ((x,x): x in M)},
and let $\mu\in ba(\widetilde{M})$ with $|\mu|(\widetilde{M})=1$ satisfy \eqref{eq: F(f)=int_Gamma f dmu for every f in Lip_0(M)}. It suffices to show that there is a subset $A$ of $M$ with $|\mu(\Gamma_{1,A})|<\delta$ and $|\mu(\Gamma_{2,A})|<\delta$ such that, for some $n\in \N$, there exists a function $f\in B_{\Lip_0(M)}$ such that $f|_{M\setminus A} = h|_{M\setminus A}$ and $f(u_n)-f(v_n)\geq (1-\delta)d(u_n,v_n)$.
Indeed, suppose that such $A$, $n$, and $f$ have been found.
Then, setting $\Gamma_A=\Gamma_{1,A}\cup\Gamma_{2,A}$, one has $|\mu|(\Gamma_A)<2\delta$,
and thus (observing that $\widetilde{f}|_{\widetilde{M}\setminus\Gamma_A}=\widetilde{h}|_{\widetilde{M}\setminus\Gamma_A}$)
\begin{align*}
F(f)
&=\int_{\widetilde{M}}\widetilde{f}\,d\mu
%=\int_{\widetilde{M}\setminus\Gamma_A}\widetilde{g}\,d\mu+\int_{\Gamma_A}\widetilde{g}\,d\mu
=\int_{\widetilde{M}}\widetilde{h}\,d\mu
                       +\int_{\Gamma_A}\bigl(\widetilde{f}-\widetilde{h}\bigr)\,d\mu\\
&=\,F(h)+\int_{\Gamma_A}\bigl(\widetilde{f}-\widetilde{h}\bigr)\,d\mu\\
&>1-2\delta-2|\mu|(\Gamma_A)
\geq 1-6\delta,
\end{align*}
and, therefore,
\[
\|F+m_{u_n,v_n}\|
\geq F(f)+\frac{f(u_n)-f(v_n)}{d(u_n,v_n)}
>2-7\delta
\geq 2-\varepsilon.
\]

It remains to find the $A$, $n$, and $f$ as above.
To this end, choose a real number $\theta\in(0,1)$ satisfying $\frac{\theta}{1-\theta}<\frac{\delta}{2}$.
Without loss of generality, one may assume that one of the following (mutually exclusive) conditions holds:
\begin{itemize}
\item[\textup(1)]
no subsequence of the sequence $(u_n)_{n=1}^\infty$ converges;
\item[\textup(2)]
there is a $u\in M$ such that $u_n=u$ for every $n\in\N$;
\item[\textup(3)]
there is a $u\in M$ with $u_n\neq u$ and $v_n\neq u$ for every $n\in\N$ such that $u_n\to u$.
\end{itemize}

(1).
In this case, by passing to a subsequence, one may assume that
there is an $r>0$ such that the open balls $A_n\coloneq B(u_n,r)$, $n=1,2,\dotsc$, are pairwise disjoint.
It follows that there is an $N\in\N$ such that, for every $n\geq N$, one has
$|\mu|(\Gamma_{1,A_n})<\delta$ and $|\mu|(\Gamma_{2,A_n})<\delta$.
Pick an $n\geq N$ so that $d(u_n,v_n)<\theta r$.
One may assume that $0\not\in A_n$.
Let $A=A_n$ and define the function $f$ by $f|_{M\setminus A}=h|_{M\setminus A}$,
$f(u_n)=h(u_n)$, $f(v_n)=h(u_n)-(1-\delta) d(u_n,v_n)$, and by extending the definition norm-preservingly to the whole $M$.

Then $\|f\|\leq1$ because, whenever $x\in M\setminus A$, one has
(taking into account that $d(u_n,v_n)\leq\frac{\delta}{2}\,d(x,v_n)$
by Lemma \ref{lem: if y in B(u,theta r) and x in M setminus B(u,r)})
\begin{align*}
|f(x)-f(v_n)|
&=|h(x)-h(u_n)+(1-\delta)\,d(u_n,v_n)|\\
&\leq (1-\delta)\,d(x,u_n)+(1-\delta)\,d(u_n,v_n)\\
&\leq  (1-\delta)\,\bigl(d(x,v_n)+2\, d(u_n,v_n)\bigr) \leq d(x,v_n).
\end{align*}

\smallskip
(2).
%Whenever $i\in\{1,2\}$, observing that the limit $\lim_{r\to0+}|\mu|(\Gamma_{i,B(u,r)})=:l_i$ exists,
%there is a real number $r_i>0$ such that $|\mu|(\Gamma_{i,B(u,r_i)})<l_i+\delta$.
%Setting $r\coloneq\min\{r_1,r_2\}$, one has
%\begin{equation}\label{eq: |mu|(Gamma_(i,B_(r,s))<delta whenever i=1 or i=2, and 0<s<r}
%\text{$|\mu|(\Gamma_{1,A})<\delta$}
%\quad\text{and}\quad
%\text{$|\mu|(\Gamma_{2,A})<\delta$}
%\qquad\text{whenever $0<s<r$,}
%\end{equation}
%where $A:=B(u,r)\setminus B(u,s)$.
%Without loss of generality, 
%One may assume that $0\not\in B(u,r)\setminus\{u\}$.
%Pick an $n\in\N$ and an $s>0$ so that $d(u, v_n)<\theta r$ and $s<\theta d(u, v_n)$.
Pick an $n\in\N$ and $r,s>0$ so that $d(v_n, u)<\theta r$, $s<\theta d(v_n, u)$, 
and $|\mu|(\Gamma_{1,A})<\delta$ and $|\mu|(\Gamma_{2,A})<\delta$ where $A=B(u,r)\setminus B(u,s)$.
One may assume that $0\not\in B(u,r)\setminus\{u\}$.
Letting $A$ be as above,
define the function $f$ by $f|_{M\setminus A}=h|_{M\setminus A}$, $f(v_n)=h(u)-(1-\delta)\, d(u, v_n)$, and by extending the definition norm-preservingly to the whole $M$.
One has $\|f\|\leq1$.
In fact, if $x\in B(u,s)$, then
(taking into account that $d(x,u)\leq\frac{\delta}{2}\,d(x,v_n)$
by Lemma \ref{lem: if y in B(u,theta r) and x in M setminus B(u,r)})
\begin{align*}
|f(x)-f(v_n)|
&=|h(x)-h(u)+(1-\delta)\,d(u, v_n)|\\
&\leq (1-\delta)\,d(x,u)+(1-\delta)\,d(u,v_n)\\
&\leq  (1-\delta)\bigl(2d(x,u)+d(x,v_n)\bigr)\leq d(x,v_n);
\end{align*}
if $x\in M\setminus B(u,r)$, then, keeping in mind that $u=u_n$,
the desired inequality $|f(x)-f(v_n)|\leq d(x,v_n)$ is obtained as in the case~(1).

\smallskip
(3). Pick an $n\in\N$ and $r,s>0$ so that 
\[
\max\{d(u, u_n), d(u, v_n)\}<\theta r,\quad s<\theta \min\{ d(u, u_n), d(u, v_n)\},
\] 
and $|\mu|(\Gamma_{1,A})<\delta$ and $|\mu|(\Gamma_{2,A})<\delta$ where $A=B(u,r)\setminus B(u,s)$.
% In this case, as in the case (2), there is  a real number $r>0$
% satisfying \eqref{eq: |mu|(Gamma_(i,B_(r,s))<delta whenever i=1 or i=2, and 0<s<r}
% where $A:=B(u,r)\setminus B(u,s)$, and $0\not\in B(u,r)\setminus\{u\}$.
% %Pick an $n\geq N$, an $r>0$, and an $s>0$ so that
% Pick an $n\in N$ and an $s>0$ so that 
% \[
% \max\{d(u, u_n), d(u, v_n)\}<\theta r\quad\text{and}\quad s<\theta \min\{ d(u, u_n), d(u, v_n)\},
% \] 
%and, by setting $A=B(u,r)\setminus B(u,s)$, one has $|\mu|(\Gamma_{1,A})<\delta$ and %$|\mu|(\Gamma_{2,A})<\delta$.
%One may assume that $0\not\in B(u,r)\setminus\{u\}$.
Letting $A$ be as above,
define the function $f$ by $f|_{M\setminus A}=h|_{M\setminus A}$, $f(u_n)=h(u)+(1-\delta)\,d(u, u_n)$, $f(v_n)=f(u_n)-(1-\delta)\,d(u_n,v_n)$, 
and by extending the definition norm-preservingly to the whole $M$.
By calculations similar to those performed in the cases (1) and~(2),
one establishes that $\|f\|\leq1$.
\end{proof}

Recall that a metric space $M$ is \emph{local} if for every $\varepsilon>0$ and for every Lipschitz function $f\colon M\to \mathbb R$ there are two distinct points $u,v\in M$ such that
$d(u, v) < \varepsilon$ and $f(m_{u,v}) > \|f\| - \varepsilon$.
The following result is an immediate consequence of the previous theorem.

\begin{cor}[see {\cite[Proposition~3.4 and Theorem~3.5]{MR3794100}}, cf. {\cite[Theorem 3.1]{MR2379289}}]
Let $M$ be a pointed metric space. If $M$ is local, then $\Lip_0(M)$ has the Daugavet property.
\end{cor}

\begin{rem}
The converse statement of this result also holds (see \cite[Proposition 2.3 and the remark following its
proof]{MR2379289} and \cite[Theorem~3.5]{MR3794100}).
On the other hand, the converse statement of our Theorem~\ref{thm: If f in Lip_0(M) is local, then f is a Daugavet point} does not hold 
since there exist uniformly discrete metric spaces $M$ for which $\Lip_0(M)$ has Daugavet points. For example, let $M$ be an infinite pointed metric space where the distance between two different elements is $1$ if one of the elements is the fixed point $0$, and the distance between two different elements is $2$ in all other cases. Then the norm-one element $f\in\Lip_0(M)$, given by $f(x)=1$ for every $x\in M\setminus\{0\}$, is a Daugavet point.
\end{rem}

%%%%%%%%%%%%%%%%%%%%%%%%%%%%%%%%%%%%%%%%%%%%%%%%%%%%%%%%%%%%%%%%%%%%%%%%%%%%%%%%%%%%%%%%%%%%%%%%%%%%%%%%%%%%%%%%%%

\section*{Acknowledgment}
The authors wish to express their thanks to Triinu Veeorg for
her active interest in the publication of this paper.

This work was supported by the Estonian Research
Council grant (PRG1901). The research of Andre Ostrak was supported by the University of Tartu ASTRA Project PER ASPERA, financed by the European Regional Development Fund.

% ----------------------------------------------------------------
\bibliographystyle{amsplain}
\bibliography{references}
%\printbibliography

\end{document}